\documentclass[12pt,a4paper]{article}
\usepackage{fullpage}
\usepackage{hyperref}

\usepackage{amsmath,amsfonts,amssymb,amsthm,amscd}
\newtheorem{Theorem}{Theorem}

\newtheorem{Lemma}{Lemma}
\newtheorem{Corollary}{Corollary}

\newcommand{\den}{\text{den}}

\newcommand{\R}{\mathbb R}
\newcommand{\Q}{\mathbb Q}
\newcommand{\Z}{\mathbb Z}
\newcommand{\N}{\mathbb N}

\def\qin{\left[n\atop k \right]}

\title{On approximation measures of $q$-exponential function}
\author{Leena Leinonen\footnote{The work of all authors was supported by the Academy of Finland, grant 138522,
and the work of L. Leinonen was also supported by Finnish Academy of Science and Letters, 
V\"ais\"al\"a Foundation.},  
Marko Leinonen  
and  Tapani Matala-aho}

\begin{document}

\maketitle
\date{}

\begin{abstract} 
We shall present effective approximations measures for certain infinite products 
related to $q$-exponential function.  There are two main targets. 
First we shall prove an explicit irrationality measure result for the values of 
$q$-exponential function at rational points. 
Then, if we restrict the approximations to rational numbers of the shape $d^s/N$, we may
replace Bundschuh's irrationality exponent $7/3$ by $2+\frac{1}{3+2\sqrt 3}=2.1547...$.
\end{abstract}

\section{Introduction}

Our work considers irrationality measures of the values $\tau=E_q(t)$ of the
$q$-exponential function 
\begin{equation}\label{EULER1}
E_q(t)=\sum_{k=0}^{\infty} \frac{t^k}{\prod_{n=1}^{k}\left(1-q^{n}\right)} = 
\prod_{k=0}^{\infty} \frac{1}{1-q^kt},\quad 0<|t|,\ |q|<1,\ t,q\in\mathbb{Q},
\end{equation}
over the field of rational numbers.

By an irrationality measure of a real number $\tau$ we mean
any function in $N$ bounding $\left| \tau-\frac{M}{N}\right|$, $M, N \in \Z$, 
from below for big enough $N$.
Irrationality exponent of a real number $\tau$ means an exponent $\mu$ for which
there exist positive constants $c$ and $N_0$ such that
\begin{equation*}\label{mudefinition}
\left| \tau-\frac{M}{N} \right| \ge \frac{c}{N^\mu}
\end{equation*}
holds for all $M, N \in \Z$, $N \ge N_0$. 
Further, the asymptotic irrationality exponent $\mu_I(\tau)$ is then the infimum of all such exponents $\mu$. 

We start our considerations by proving a fully explicit irrationality measure for
the $q$-exponential function in the case 
$1/q\in\mathbb{Z}\setminus\{0,\pm 1\}$.

\begin{Theorem}\label{Theorem1}
Let $q=1/d$, where $|d| \in \Z_{\ge 2}$, and $t=u/v\in\Q$, where $u\in \Z$, $v\in \Z_+$,
gcd$(u,v)=1$ and  $0<|t|<1$. 
Then 
\[
\left|E_q (t)- \frac{M}{N}\right| \ge C_1(2|N|)^{-\left(\frac{7}{3}+\varepsilon_1\right)} 
\]
holds for all  $M,N\in \Z$, $|N|\ge 1$, with an explicit constant 
$C_1= C_1(d,v)$ (given later in \eqref{CONSTANTC})
and
\[
\varepsilon_1=\varepsilon_1(v,d, N)=\frac{\frac{22}{3}\log v + 8\log|d|+4\sqrt{(\log 4)\log |d|}}{\sqrt{\frac{3}{2}(\log|d|)\log(2|N|)}} .
\]
\end{Theorem}
In particular, this means that the (asymptotic) irrationality exponent of 
\\
$\prod_{k=1}^{\infty}\left(1-d^{-k}\right)$ is at most $7/3$, a record holding
result proved already by Bundschuh \cite{bundshuh}. 
However, Bundschuh's result is not as explicit as ours.
The work \cite{matvaa} generalizes Bundschuh's result over arbitrary number fields
and offers similar results for several other $q$-series related to the $q$-exponential
function (\ref{EULER1}).

Our work is partly inspired by the work \cite{lucmarsta} where the authors are interested in
$C$-nomials including the $q$-binomial coefficients and the Fibonomials defined by
$$\qin_q = \frac{(1-q)(1-q^2)\cdots(1-q^n)}{(1-q)(1-q^2)\cdots(1-q^k)(1-q)(1-q^2)\cdots(1-q^{n-k})}$$
and
\[\left[{n}\atop{k}\right]_F= \frac{F_1\cdots F_n}{F_1\cdots F_k\cdot F_1\cdots F_{n-k}}\]
respectively, where $F_k$ are the Fibonacci numbers, $F_0=0, F_1=1$, $F_{k+2}=F_{k+1}+F_k$ for $k=0,1,...$.

Connected to the investigations of distances of $q$-binomial coefficients, 
the authors in \cite{lucmarsta} state the following result.
Let $q=1/d$, where $d \in \Z_{\ge 2}$, then  
\begin{equation}\label{LUCA1}
\left|\frac{1}{\prod_{n=1}^{\infty}\left(1-d^{-n}\right)}-\frac{d^s}{(d-1)^l}\right|>
\frac{1}{\left((d-1)^{l}\right)^3}
\end{equation}
holds for all but finitely many pairs $(l,s)\in\mathbb{Z}_{\ge 1}^2$.
As noted in \cite{lucmarsta}, already Bundschuh's result, see \cite{bundshuh}, is better than (\ref{LUCA1}).
However, we can do more in this restricted case. Namely, we have 
\begin{Theorem}\label{Theorem2}
Let $q=1/d$, where $|d| \in \Z_{\ge 2}$, and $t=u/v\in\Q$, where $u\in \Z$, $v\in \Z_+$,
gcd$(u,v)=1$ and  $0<|t|<1$. 
Then there exists an effective positive constant $C_2=C_2(d,v)$ such that
\begin{equation}\label{LMT2}
\left|E_q (t)- \frac{d^s}{N}\right|
 \ge C_2 (2|N|)^{-(2+\frac{1}{3+2\sqrt 3}+\varepsilon_2)}
\end{equation}
holds for all $s\in \Z_+, N\in\Z$, $|N| \ge 82836$, with
$\varepsilon_2=\varepsilon_2(d,v,N)\in\R_+$ satisfying 
$\varepsilon_2(d,v,N)\to 0$, when $|N|$ increases.
\end{Theorem}
Consequently, (\ref{LMT2}) implies that in this restricted case 
the (restricted) asymptotic irrationality exponent is at most 
$2+\frac{1}{3+2\sqrt 3}=2.1547...$  
which is smaller than Bundschuh's result $7/3$.

\section{A lemma for irrationality exponents}

Our target is to approximate the values $E_q(t)$ of the $q$-exponential series by rational 
numbers of the shape $M/N$ in the following cases: 
	\begin{itemize}
	\item[(a)] $M,N\in \Z, N\ne 0$,
	\item[(b)] $M=d^s$, $s\in \Z_+$, $N\in \Z, N\ne 0$,
	\end{itemize}
where $q=1/d$, $|d|\in\Z_{\ge 2}$ and $0<|t|<1$, $t\in\mathbb{Q}$.
The cases (a) and (b) correspond to Theorems \ref{Theorem1} and \ref{Theorem2}, respectively. 

Next we present a result which will be applied for deducing irrationality measures and
upper bounds of irrationality exponents.

\begin{Lemma}\label{LBLEMMA}
Let $\Phi\in\mathbb{R}$. 
Assume that we have a sequence of numerical linear forms
\begin{equation}\label{numeforms}
q_n\Phi-p_n=r_n,\quad q_n,\ p_n\in\mathbb{Q},\quad n\in\mathbb{N},
\end{equation}
satisfying the conditions
\begin{align}
&q_np_{n+1}-p_nq_{n+1}\ne 0,\label{DET}\\
&|q_n|,|p_n|\le Q(n)=e^{an^2+a_1n+a_2},\label{Qbound}\\
&|r_n|\le R(n)=e^{-bn^2+b_1n+b_2}<1\label{Rbound}
\end{align}
for all $n\ge n_0\ge 1$ with some constants $a, a_1,a_2, b,b_1,b_2\in\mathbb{R}_{\ge 0}$, $a,b>0$.
Denote $c_1=(b_1+\sqrt{b_1^2+4bb_2})/2b$, $c_2=2ac_1+4a+a_1$,  and $c_3=c_1(ac_1+4a+a_1)+4a+2a_1+a_2$.
Then 
\begin{equation}\label{LOWERBOUND}
\left| \Phi-\frac{M}{N} \right| >\frac{1}{e^{c_3}}
 \frac{1}{(2|N|)^{1+a/b+c_2/\sqrt{b\log(2|N|)}}}
\end{equation}
for all $(M,N)\in\mathbb{Z}^2\cap \mathcal{E}$, where $\mathcal{E}$ is the set defined below.\\
Write $D_{n}:=Np_n-Mq_n$ and define a set $\mathcal{E}$:
Let $M,N\in\mathbb{Z}$. Then $(M,N)\in \mathcal{E}$, if
\begin{equation*}\label{NN0}
|N|\ge N_0:=e^{bn_0^2-b_1n_0-b_2}/2,
\end{equation*}
and if there exists  a largest $\bar{n}\in \Z_+$, $\bar{n}\ge n_0$, 
such that
\begin{equation}\label{Rbarn}
2|N|R(\bar n)\ge 1
\end{equation}
and
\begin{equation}\label{KOKOLUKU}
D_{\bar n+1}, D_{\bar n+2}\in \mathbb{Z}.
\end{equation}
\end{Lemma}

Of course, \eqref{KOKOLUKU} is satisfied for all $M,N\in\mathbb{Z}$, $N\ne 0$,
if $p_n,q_n\in\mathbb{Z}$. Thus \eqref{Rbarn} will also be valid because $R(n)\to 0$ as $n\to \infty$.
This will be the case in Theorem \ref{Theorem1}.
In Theorem \ref{Theorem2} the condition \eqref{KOKOLUKU} plays a crucial role for improving the
the lower bound \eqref{LOWERBOUND}.
However, the restricted approximations $M/N$ satisfying the condition 
\eqref{KOKOLUKU} may effect \eqref{Rbarn}, see Lemma \ref{EFFECT}.

\begin{proof}
By denoting 
$
\Lambda:=N\Phi-M
$
and using \eqref{numeforms} we get 
\begin{equation}\label{qnLambda}
q_{n}\Lambda =Nr_{n}+D_{n}.
\end{equation}
Suppose $(M,N)\in E$. Then there exists  a largest $\bar{n}\in \Z_+$, $\bar{n}\ge n_0$, 
such that \eqref{Rbarn} is true.

Further, by \eqref{DET} and \eqref{KOKOLUKU} we know that $D_n\in\mathbb{Z}\setminus\{0\}$, where
$n=\bar n +1$ or $n=\bar n+2$. 
Thus by \eqref{qnLambda}, \eqref{Qbound} and  \eqref{Rbound} we get 
\[
1\le |D_n|=|q_{n}\Lambda-Nr_{n}|\le |q_{n}||\Lambda| + |N||r_{n}|\le Q(n)|\Lambda |+|N|R(n),
\]
where $|N|R(n)<\frac{1}{2}$ by \eqref{Rbarn}. 
Hence 
\begin{equation}\label{Lambdaestimate}
1<2|\Lambda|Q(n)\le 2|\Lambda|e^{an^2+a_1n+a_2}.
\end{equation}
From \eqref{Rbound} and \eqref{Rbarn} we get 
\begin{equation*}
b\bar n^2-b_1\bar n-(b_2+\log(2|N|))\le 0
\end{equation*}
and further
\begin{equation}\label{barnb2}
\bar n\le \frac{b_1+\sqrt{b_1^2+4bb_2+4b\log(2|N|)}}{2b}\le c_1+\frac{\sqrt{\log(2|N|)}}{\sqrt{b}},
\end{equation}
where we applied the inequality $\sqrt{A+B}\le \sqrt{A}+\sqrt{B}$, $A,B\ge 0$.
Then the estimate \eqref{barnb2} implies
\begin{align*}
&an^2+a_1n+a_2
\le a(\bar n+2)^2+a_1(\bar n+2)+a_2 \\
\le &\frac{a}{b}\log(2|N|)+\frac{c_2}{\sqrt{b}}{\sqrt{\log(2|N|)}}+c_3\\
\end{align*}
which with \eqref{Lambdaestimate} proves \eqref{LOWERBOUND}. 
\end{proof}

The exponent $1+a/b+c_2/\sqrt{b\log(2|N|)}$ in \eqref{LOWERBOUND} 
gives an upper bound for the irrationality exponent of $\Phi$. 
In the following we call $1+a/b$ the main term and 
$\varepsilon(N):=c_2/\sqrt{b\log(2|N|)}$ will be called the error term.

\section{Pad\'e -approximations}

First we give definitions of $q$-series factorials 
$$(a)_0=1,\
(a)_n=(1-a)(1-aq)\cdots (1-aq^{n-1}),\quad n\in \Z_+,$$
and the $q$-binomial coefficients
$$\qin =\qin_q= \frac{(q)_n}{(q)_k (q)_{n-k}},\quad 0\le k\le n.$$
Note that
\begin{equation*}
\qin_q \in \Z[q],\qquad \deg_q \qin_q =k(n-k),
\end{equation*}
see e.g \cite{andrews}.

Our starting point is the following Pad\'e approximation formula of $q$-exponential function from
\cite[Article VI, Lemma1]{matala}. 

\begin{Lemma}[\cite{matala}]\label{padet}
Let 
\[
B_n(t)=\sum_{k=0}^{n}\left[{n}\atop{k}\right]q^{\binom{k}{2}}(q^{n+1})_{n-k}(-t)^k,
\]
\[
A_n(t)=\sum_{k=0}^{n}\left[{n}\atop{k}\right]q^{kn}(q^{n+1})_{n-k}t^k
\]
and
\[
S_n(t)=(-1)^{n}t^{2n+1}q^{\frac{3n^2+n}{2}}\frac{(q)_n}{(q)_{2n+1}}\sum_{k=0}^{\infty }\frac{(q^{n+1})_k}{(q)_k(q^{2n+2})_k}t^k.
\]
Then 
\begin{equation}\label{pade}
B_n(t) E_q(t)-A_n(t)=S_n(t).
\end{equation}
\end{Lemma}

As usual when studying Diophantine properties of $q$-series we need to accelerate convergence
of Pad\'e approximations \eqref{pade}. This will be done by iterate use of 
the $q$-shift operator $J$, $JF(t)=F(qt)$, and application of the $q$-difference equation 
\[
E_q(qt)=(1-t)E_q(t).
\]
First we obtain
\[
(1-t)B_n(qt) E_q(t)-A_n(qt)=S_n(qt)
\]
and repeating this process $K$ times we get our new approximations
\begin{equation}\label{JK}
(t)_KB_n(q^Kt) E_q(t)-A_n(q^Kt)=S_n(q^Kt).
\end{equation}

The Pad\'e approximations suggest that
\[
 E_q(t)\sim\frac{A_{n}(q^Kt)}{(t)_KB_{n}(q^Kt)}.
\]
Thus our starting point will be the following expression
\begin{equation}\label{omega}
\Omega_{n,K} (q,t) =NA_{n}(q^Kt)- M (t)_K B_{n}(q^Kt).
\end{equation}

\section{Denominators}

For any rational number $\eta\in \Q$ we call 
\[\den(\eta)= \min\{d\in\Z_+ |\, d\eta \in \Z\}\]
the denominator of $\eta$. Thus, e.g. we have $v=\den(t)$.
In the  following lemma we give estimates of the denominators of $\Omega_{n,K}(q,t)$ in our two cases.

\begin{Lemma} \label{denD}
\begin{itemize}
	\item[(a)] We have $\den(\Omega_{n,K}(q,t))\le v^{n+K}|d|^{\delta_{n,K}}$, where 
	\[
	\delta_{n,K} =\begin{cases}
	\frac{K^2-K}{2}+ \frac{3n^2+n}{2}, \quad \text{when} \quad 0\le K \le n;\\
	\frac{K^2-K}{2}+ \frac{n^2-n}{2}+nK, \quad  \text{when} \quad K > n. \end{cases}
	\] 
	\item[(b)]  We have $den(\Omega_{n,K}(q,t))\le v^{n+K}|d|^{\delta_{n,K}}$, 
where  
	\[
	\delta_{n,K} =
\begin{cases}
	\frac{K^2-K}{2}+ \frac{3n^2+n}{2}, \quad \text{when} \quad 0\le K \le n;\\
	n^2+nK, \quad  \text{when} \quad K > n,
    \end{cases}
\]
if the assumption
\begin{equation}\label{sKKnn}
s \ge \frac{K(K-1)}{2}-\frac{n(n+1)}{2}
\end{equation}
holds in the case $K>n$.	
\end{itemize} 
\end{Lemma}

\begin{proof}
By Lemma \ref{padet} we have 
\[
A_{n}(q^Kt)=\sum_{k=0}^{n}\left[{n}\atop{k}\right]q^{kn}(q^{n+1})_{n-k}(q^{K}t)^k,\quad
B_{n}(q^Kt)=\sum_{k=0}^{n}\left[{n}\atop{k}\right]q^{\binom {k}{2}}(q^{n+1})_{n-k}(-q^{K}t)^k,
\]
where $\left[{n}\atop{k}\right], (q^{n+1})_{n-k} \in \Z[q]$ and  
\[\deg_q\left[{n}\atop{k}\right] = k(n-k), \quad 
\deg_q (q^{n+1})_{n-k}= \frac{3n^2}{2}-2kn+\frac{k^2}{2}+\frac{n}
{2}-\frac{k}{2}. 
\] 
Thus we see immediately that $A_{n}(q^Kt), (t)_KB_{n}(q^Kt) \in \Z[q,t]$ and
\begin{equation}\label{deg_t}
\deg_t\,A_{n}(q^Kt)=n, \,\, \deg_t\,(t)_K B_{n}(q^Kt)=n+K.
\end{equation}

Next we estimate the degrees $\deg_q\,A_{n}(q^Kt)$ and $\deg_q\,(t)_KB_{n}(q^Kt)$. Now 
\[\begin{split}
\deg_q A_{n}(q^Kt) &=\max_{0\le k\le n} \{\deg_q \left[{n}\atop{k}\right]+kn+ \deg_q(q^{n+1})_{n-k}+ kK\}\\
&=\max_{0\le k\le n} \{\frac{3n^2}{2}+\frac{n}{2}-\frac{k^2}{2}-\frac{k}{2}+kK\}.
\end{split}\]
If  $K\le n$, we obtain the maximum, when $k=K$. On the other hand, if $K> n$, we obtain the maximum, when $k=n$. Hence we have
\begin{equation}\label{deg_q_A}
\deg_q A_{n}(q^Kt)\le \begin{cases}
	\frac{K^2-K}{2}+ \frac{3n^2+n}{2}, \quad \text{when} \quad 0\le K \le n;\\
	n^2+nK, \quad  \text{when} \quad K > n. \end{cases} 
	\end{equation}
Similarly we obtain that
\[\begin{split}
\deg_q (t)_K B_{n}(q^Kt) &=\max_{0\le k\le n} \{ \deg_q (t)_K+\deg_q \left[{n}\atop{k}\right]+\binom{k}{2}+ \deg_q(q^{n+1})_{n-k}+ kK\}\\
&=\max_{0\le k\le n}\{\frac{K(K-1)}{2}+ \frac{3n^2+n}{2}+  k(K-n-1)\}.
\end{split}\]
If $K\le n$ we obtain the maximum when $k=0$ and if  $K> n$ we obtain the maximum at $k=n$. Hence we have
\begin{equation}\label{deg_q_B}
\deg_q (t)_KB_{n}(q^Kt)\le \begin{cases}
	\frac{K^2-K}{2}+\frac{3n^2+n}{2}, \quad \text{when} \quad 0\le K \le n;\\
	\frac{K^2-K}{2} + \frac{n^2-n}{2}+nK, \quad  \text{when} \quad K > n. \end{cases} 
\end{equation}
\begin{itemize}
\item[(a)]We have $N, M\in\Z$, $v=\text{den}(t)$ and $q=1/d$. 
Thus by \eqref{omega}, \eqref{deg_t} - \eqref{deg_q_B} we obtain that $\text{den}(\Omega_{n,K}(q,t))\le v^{n+K}|d|^{\delta_{n,K}}$, where
\[\delta_{n,K}=\begin{cases}\frac{K^2-K}{2}+ \frac{3n^2+n}{2}, \quad \text{ when} \quad 0\le K \le n;\\ 
\frac{K^2 - K}{2}+ \frac{n^2-n}{2}+nK, \quad  \text{when} \quad K > n. \end{cases}\]

\item[(b)]
 Here $N, M=d^s \in \Z$, $v= \text{den}(t)$ and $q=1/d$. Thus 
$$\text{den}(M(t)_KB_{n}(q^Kt))\le v^{n+K}|d|^{\hat\delta_{n,K}},$$
where
 \begin{equation}\label{deg_q_MB}
 \hat\delta_{n,K} 
\le \begin{cases}
\frac{K^2-K}{2}+\frac{3n^2+n}{2}-s, \quad \text{when} \quad 0\le K \le n;\\
\frac{K^2-K}{2} + \frac{n^2-n}{2}+nK -s, \quad  \text{when} \quad K > n.
 \end{cases}
\end{equation}
In this case we assumed that $s\ge \frac{K^2-K}{2}- \frac{n^2+n}{2}$, hence by \eqref{omega}, \eqref{deg_t}, \eqref{deg_q_A} and \eqref{deg_q_MB} we get
$\text{den}(\Omega_{n,K}(q,t))\le v^{n+K}|d|^{\delta_{n,K}}$, where
\begin{equation*}
\delta_{n,K} 
\le\begin{cases}
	\frac{K^2-K}{2}+\frac{3n^2+n}{2}, \quad \text{when} \quad 0\le K \le n;\\
	n^2 +nK, \quad  \text{when} \quad K > n. 
    \end{cases}
\end{equation*}
\end{itemize}

\end{proof}

\section{Numerical approximations and a non-vanishing result}

By using the notations
\begin{equation}\label{p_and_q}\begin{split}
p_{n,K}(q,t)&:= \text{den}(\Omega_{n,K}(q,t))A_{n}(q^Kt),\\ 
q_{n,K}(q,t)&:= \text{den}(\Omega_{n,K}(q,t))(t)_KB_{n}(q^Kt),\\
r_{n,K}(q,t)&:= \text{den}(\Omega_{n,K}(q,t))S_{n}(q^Kt)\\ 
\end{split}\end{equation}
we get numerical approximations
\begin{equation}\label{def_r_nK}
r_{n,K}(q,t)=q_{n,K}(q,t)E_q(t)-p_{n,K}(q,t),
\end{equation}
where $p_{n,K}(q,t)$ and $q_{n,K}(q,t)$ are so called polynomial terms and
$r_{n,K}(q,t)$ is the remainder term.
Here we note that the following quantity
\begin{equation}\label{DnK}
D_{n,K}(q,t):=\text{den}(\Omega_{n,K}(q,t))\cdot\Omega_{n,K}(q,t)=
p_{n,K}(q,t)N-q_{n,K}(q,t)M,
\end{equation}
is an integer which plays a crucial role later.

Next we prove a non-vanishing result.
\begin{Lemma}
For the  Pad\'e-approximation polynomials of Lemma \ref{padet} holds that
\[\Delta_n(t)= \left|\begin{matrix} B_n(t) & A_n(t)\\ B_{n+1}(t) & A_{n+1}(t)\end{matrix}\right|= 
C_n t^{2n+1},\]
where $C_n=(q^{n+2})_{n+1}(-1)^nq^{\frac{3n^2+n}{2}}\frac{(q)_n}{(q)_{2n+1}}$ is a non-zero if
$0<|q|<1$.
\end{Lemma}
\begin{proof}
We have
\[
\Delta_n(t)=B_n(t)A_{n+1}(t)-B_{n+1}(t)A_n(t)
\]
and by Lemma \ref{padet} we know that $\deg_t A_n(t), \deg_t B_n(t)\le n$. Hence we get that \[\deg_t\Delta_n(t)\le 2n+1.\] 
On the other hand, it follows from Pad\'e-approximation formula \eqref{pade} that
\[
\begin{cases}
B_n(t)E_q(t)-A_n(t) =t^{2n+1}\hat{S}_n(t);\\
B_{n+1}(t)E_q(t)-A_{n+1}(t) =t^{2n+3}\hat{S}_{n+1}(t),
\end{cases}
\]
where 
\[
\hat{S}_n(t)=(-1)^nq^{\frac{3n^2+n}{2}}\frac{(q)_n}{(q)_{2n+1}}\sum_{k=0}^{\infty}\frac{(q^{n+1})_k}{(q)_k(q^{2n+2})_k}t^k.
\]
Thus 
\begin{align*}
\Delta_n(t) &= \left|\begin{matrix}B_n(t) & -t^{2n+1}\hat{S}_n(t)\\ B_{n+1}(t) & -t^{2n+3}\hat{S}_{n+1}(t) \end{matrix}\right| \\
&=t^{2n+1}\left(B_{n+1}(t)\hat{S}_n(t)-t^2B_n(t)\hat{S}_{n+1}(t)\right)
\end{align*}
which implies that $\text{ord}_t \Delta_n(t) \ge 2n+1$ and 
\[
\Delta_n(t)= C_nt^{2n+1},
\]
where
\[C_n= B_{n+1}(0)\hat{S}_n(0)= (q^{n+2})_{n+1}(-1)^nq^{\frac{3n^2+n}{2}}\frac{(q)_n}{(q)_{2n+1}}.\] 
\end{proof}

\begin{Corollary}\label{ne0}
Suppose $0<|t|<1$, $0<|q|<1$. Then
$$\Delta_{n,K}(q,t)
= \left|\begin{matrix}q_{n,K}(q,t) & p_{n,K}(q,t) \\ q_{n+1,K}(q,t) & p_{n+1,K}(q,t)\end{matrix}\right| \ne 0
$$
for all $n,K\in \N$. 
\end{Corollary}
\begin{proof}
Due to \eqref{p_and_q}
\[
\begin{split}
\Delta_{n,K}&(q,t) =q_{n,K}(q,t)p_{n+1, K}(q,t)-p_{n,K}(q,t)q_{n+1,K}(q,t)\\
&=\text{den}(\Omega_{n,K}(q,t))\text{den}(\Omega_{n+1,K}(q,t))(t)_K\left(B_n(q^{K}t)A_{n+1}(q^{K}t)-A_n(q^{K}t)B_{n+1}(q^{K}t)\right). 
\end{split}
\]
Thus we have $\Delta_{n,K}(q,t)=\text{den}(\Omega_{n,K}(q,t))\text{den}(\Omega_{n+1,K}(q,t))(t)_K\Delta_n(q^{K}t)$, which is non-zero for all $n,K\in \N$ and $0<|t|<1$, $0<|q|<1$.
\end{proof}

\section{Estimates}

In the following we give estimates for the upper bounds of $|p_{n,K}(q,t)|$ and $|q_{n,K}(q,t)|$. 
\begin{Lemma}\label{Qn}
Let $n,K\in \Z_+$, then
 \[
 \max{\{|q_{n,K}(q,t)|, |p_{n,K}(q,t)|\}}\le Q_1(n,K)=8\,\max{\{1,(t)_K\}} \den(\Omega_{n,K}(q,t)).
 \]
\end{Lemma}
\begin{proof}
According to \cite[Article VI, p.8-9]{matala} we have
\[1-(q+q^2) < (q)_k \le 1 \quad \text{for all}\quad 0<q < 1,\,\, k\in \N\]
and
\[1 \le (q)_k < 1+ |q| \quad \text{ for all} \quad \frac{1-\sqrt{5}}{2} < q < 0, \,\, k\in \N.\]
Because $q=1/d$, $|d|\in \Z_{\ge 2}$, we have $\frac{1}{4} <(q) _k \le1$ for positive $q$, and $1 \le (q)_k < \frac{3}{2}$ for negative $q$. 
Hence by \eqref{p_and_q} and Lemma \ref{padet} we obtain that
\[\begin{split}
\max \{|q_{n,K}(q,t)|, |p_{n,K}(q,t)|\} &=\text{den}(\Omega_{n,K}(q,t))\max \{|A_{n}(q^Kt)|, |(t)_KB_{n}(q^Kt)|\}\\
 &\le 4\,\text{den}(\Omega_{n,K}(q,t))\max\{1,(t)_K\}\sum_{k=0}^\infty(|q|^{K}|t|)^k\\ 
&\le 8\,\max\{1,(t)_K\}\,\text{den}(\Omega_{n,K}(q,t)), 
\end{split}\]
where $(t)_K$ (with $0<|t|<1$) has the upper bound
\begin{equation*}
(t)_K < \prod_{k=0}^\infty \left(1 +\frac{1}{2^k}\right) \sim 4,768462.
\end{equation*}

\end{proof}


Above we have considered so called polynomial terms of our approximation. 
Next we will concentrate on the remainder term
$r_{n,K}(q,t)$.

\begin{Lemma}\label{Rn}
For our remainder term holds the estimate $|r_{n,K}(q,t)|\le R(n,K)$ with \\
$R(n,K)=8|d|^{-\omega_{n,K}}$ where
\begin{itemize}
\item[(a)]
$\omega_{n,K} = \begin{cases} 2nK-\frac{K^2}{2}-(n+K)\frac{\log v}{\log|d|},\quad\text{when}\, K\le n;\\
 n^2+nK-\frac{K^2}{2}-(n+K)\frac{\log v}{\log|d|},\quad\text{when}\, K> n,\end{cases}$
\item[(b)]
$\omega_{n,K} = \begin{cases}
2nK-\frac{K^2}{2}-(n+K)\frac{\log v }{\log|d|},\quad\text{when}\, K\le n;\\
     \frac{n^2}{2}+nK-(n+K)\frac{\log v}{\log|d|},\quad\text{when}\, K > n,\end{cases}$\\ 
if the assumption
$s> \frac{K^2-K}{2}-\frac{n^2+n}{2}$ holds in the case $K>n$.	
\end{itemize}

\end{Lemma}

\begin{proof}
Due to Lemma \ref{padet}, \eqref{JK}, \eqref{p_and_q} and \eqref{def_r_nK} we have
\[\begin{split}
|r_{n,K}(q,t)|
&=\text{den}(\Omega_{n,K}(q,t))|S_n(q^{K}t)|\\
         &\le\text{den}(\Omega_{n,K}(q,t))(|q|^{K}|t|)^{2n+1}|q|^{\frac{3n^2+n}{2}}\frac{(q)_n}{(q)_{2n+1}}\sum_{k=0}^\infty \frac{(q^{n+1})_k(|q|^{K}|t|)^k}{(q)_k(q^{2n+2})_k}.
\end{split}\]

Because $0<|q| \le\frac{1}{2}$, we have  $\frac{1}{4}< (q)_k\le 1$ for positive $q$, and $ 1 \le (q)_k < \frac{3}{2}$ for negative $q$.  Additionally $0<|t|<1$. From these facts it follows that
\begin{align*}
|r_{n,K}(q,t)| &< 4\,\text{den}(\Omega_{n,K}(q,t))|q|^{\frac{3n^2+n}{2}+2nK+K}\sum_{k=0}^\infty(|q|^{K}|t|)^k\\
&\le 8\,\text{den}(\Omega_{n,K}(q,t))|q|^{\frac{3n^2+n}{2}+2nK+K},
\end{align*}
where $\text{den}(\Omega_{n,K}(q,t))= v^{n+K}|d|^{\delta_{n,K}}$. 
Because $d = q^{-1}$, we get the upper bounds of the lemma directly from Lemma \ref{denD}.

\end{proof}

\section{Proof of Theorems \ref{Theorem1} and \ref{Theorem2}}

Due to \eqref{JK} and \eqref{p_and_q} we have numerical linear forms
\[q_{n,K}E_q(t)-p_{n,K} = r_{n,K} \quad (n\in \Z_+), \]
where $q_{n,K}, p_{n,K}\in \Z$ and, by Corollary \ref{ne0}, \eqref{DET} is satisfied.
From now on we use the notation $\gamma=K/n$.
Then Lemmas \ref{denD} and \ref{Qn} imply that
\begin{equation*}
\max\{|q_{n,K}(q,t)|, |p_{n,K}(q,t)|\}\le Q_1(n)=e^{an^2+a_1n+a_2},
\end{equation*}
where 
\[
a_1=a_1(\gamma)= \begin{cases}(\frac{1}{2}- \frac{\gamma}{2})\log |d| +(\gamma +1)\log v, \quad 0 < \gamma \le 1;\\ 
								(\gamma +1)\log v,\quad  \gamma >1, 
							\end{cases}
								\]
and
\[
a_2= \log (8\prod_{k=0}^\infty(1 +\frac{1}{2^k})).
\]
In the case (a)
\[
\frac{a}{\log|d|}=\frac{a(\gamma)}{\log|d|}=
\begin{cases}
			\frac{\gamma^2}{2}+ \frac{3}{2}, \quad 0 < \gamma \le 1;\\ 
 			\frac{\gamma^2}{2} + \gamma + \frac{1}{2},\quad  \gamma >1, 
\end{cases}
\]
and 	in the case (b) we have
\[
\frac{a}{\log|d|}=\frac{a(\gamma)}{\log|d|}=
\begin{cases} 
			\frac{\gamma^2}{2}+ \frac{3}{2}, \quad 0 < \gamma \le  1;\\ 
			1 + \gamma, \quad \gamma> 1.
\end{cases}
\]

Lemma \ref{Rn} implies that
\begin{equation*}
|r_{n,K}(q,t)|\le R(n,K)=e^{-bn^2+b_1n+b_2},  
\end{equation*}
where $b_1=b_1(\gamma)=(1+\gamma) \log v$ and $b_2=\log 8$. In the case (a)
\[
\frac{b}{\log|d|}=\frac{b(\gamma)}{\log|d|}=
\begin{cases} 
                2\gamma -\frac{\gamma^2}{2}, \quad 0<\gamma\le 1;\\
                1+ \gamma -\frac{\gamma^2}{2}, \quad 1 < \gamma < (1+\sqrt{3}), 
\end{cases}
\]
and in the case (b)
\[
\frac{b}{\log|d|}=\frac{b(\gamma)}{\log|d|}= \begin{cases}2\gamma -\frac{\gamma^2}{2}, \quad 0<\gamma\le 1;\\                    
	               \frac{1}{2} + \gamma, \quad \gamma > 1. 
			\end{cases}
\]

Due to Lemma \ref{LBLEMMA} the main term of our irrationality exponent is $\mu=1+a/b$. 
In the following we fix $\gamma$ so that the quantity
$a(\gamma)/b(\gamma)$ will be least possible. 

Proof of Theorem \ref{Theorem1}.
In the case (a), we have 
\[\frac{a(\gamma)}{b(\gamma)}= \begin{cases}
\frac{\gamma^2 +3}{4\gamma-\gamma^2}, \quad 0 < \gamma < 1;\\
\frac{\gamma^2 + 2\gamma +1}{2+ 2\gamma -\gamma^2}, \quad 1\le \gamma < 1 +\sqrt{3}.\end{cases}\]
Obviously $\frac{a(\gamma)}{b(\gamma)}$ is continuous, and it is decreasing when $0< \gamma < 1$ and increasing when $1 \le \gamma < 1+\sqrt{3}$. Thus
\[\min_{0 < \gamma < 1+\sqrt{3}}\frac{a(\gamma)}{b(\gamma)}= \frac{a(1)}{b(1)}= \frac{4}{3},\]
and we get the lower bound \eqref{LOWERBOUND} with the constants
\begin{equation*}\label{}
\begin{cases}
c_1=(b_1+\sqrt{b_1^2+4bb_2})/2b,\quad bc_1^2-b_1c_1-b_2=0,\\
c_2=2ac_1+4a+a_1,\\
c_3=c_1(ac_1+4a+a_1)+4a+2a_1+a_2=\left(\frac{ab_1}{b}+4a+a_1\right)c_1+\frac{ab_2}{b}+4a+2a_1+a_2.
\end{cases}
\end{equation*}
By using the estimate
\begin{equation*}\label{}
c_1\le \frac{b_1}{b}+\sqrt{\frac{b_2}{b}}
\end{equation*}
we have
\begin{equation}\label{CONSTANTC}
\begin{aligned}
C_1 &=e^{-c_4},\\
 c_3 &\le c_4:=
\frac{14\log v}{3\log |d|} \left(\frac{4}{3}\log v+\sqrt {\log 4 \log |d|}\right)\\
&\qquad\qquad+8\sqrt {\log 4\log |d|}
 +\log (2^7v^{\frac{44}{3}}|d|^8\prod_{k=0}^\infty(1 +\frac{1}{2^k}));\\
\varepsilon_1(v,d,N) &= \frac{c_2}{\sqrt{b\log(2N)}} \le \frac{\frac{22}{3}\log v+ 8\log|d|+4\sqrt{2\log 2 \log|d|}}{\sqrt{\frac{3}{2}\log|d|\log(2N)}}.
\end{aligned}
\end{equation}
This proves Theorem \ref{Theorem1}.

Proof of Theorem \ref{Theorem2}.
In the restricted case (b), we have 
\begin{equation}\label{IRRATEXP1b}
\frac{a(\gamma)}{b(\gamma)} = 
\begin{cases}
\frac{\gamma^2 +3}{4\gamma-\gamma^2}, \quad 0 < \gamma \le 1;\\
\frac{1+\gamma}{\frac{1}{2} + \gamma}, \quad \gamma > 1,
\end{cases}
\end{equation}
which is a continuous and decreasing function for all $\gamma >0$.

Because the main term of irrational exponent is equal in the case (a) and (b), 
when $0<\gamma \le 1$, we need to consider the restricted case (b) only when $\gamma > 1$.
Remember our assumptions $|d|\ge 2$, $0<|t|<1$.
Thus 
\[\tau=E_q(t)\ge (-1, \frac{1}{2})_\infty^{-1} > \frac{1}{5},\]
and by Lemma \ref{Rn} we may write
$$R(n,K)= |d|^{-(\frac{n^2}{2}+nK-(n+K)\frac{\log(v)}{\log|d|}-3)},\quad K>n,$$ 
if
$s \ge \frac{K^2 - K}{2}- \frac{n^2 +n}{2}.$
Our target is to approximate $\tau = E_q(t)$ by numbers of the shape $d^s/N$.
This implies that, when $|N|$ grows then also $s$ grows and 
thus the conditions \eqref{Rbarn} and \eqref{KOKOLUKU} will have a complicated association.
In the following lemma we will study the set $\mathcal{E}$ defined in Lemma \ref{LBLEMMA}.
In Lemma \ref{EFFECT} we are interested in enough close approximations, say
$\left|\tau- d^s/N\right|< 1/|N|$, and thus we may use approximations with $d^sN>0$.

\begin{Lemma}\label{EFFECT} 
Assume that $\gamma>1$ and 
\begin{equation}\label{lahella}
\left|\tau- \frac{d^s}{N}\right|< \frac{1}{|N|}.
\end{equation}
Then $(d^s,N)\in \mathcal{E}$, if $|N|\ge N_2$ (given below in \eqref{ISONalaraja2}) and
\begin{equation}\label{gammaN}
\begin{aligned}
\sqrt{\gamma^2-1}&\left(2+x+(4+x)\gamma+
\sqrt{(1+\gamma)^2x^2+(1+2\gamma)(6+T)}\right)\le (1+2\gamma)\sqrt{T},\\
 &x:=\frac{\log v}{\log|d|},\quad T:=\frac{2\log(|N|)}{\log|d|}.
\end{aligned}
\end{equation}
\end{Lemma}

\begin{proof}
We want choose a $n$ satisfying the conditions
\begin{equation}\label{jaannosehto}
R(n,K)<\frac{1}{2|N|},\quad s \ge \frac{K^2-n^2}{2}-\frac{K+n}{2}.
\end{equation}
Note that the first condition in \eqref{jaannosehto} is equivalent to
\begin{equation*}\label{log2Ncondition}
\frac{\log(2|N|)}{\log|d|}< (1/2+\gamma)n^2-(1+\gamma)xn-3,\quad x=\frac{\log v}{\log|d|}.
\end{equation*}
Therefore we define
\begin{equation*}\label{barngamma}
\bar n:=\left\lfloor \frac{(1+\gamma)x+\sqrt{(1+\gamma)^2x^2+(1+2\gamma)(6+T})}{1+2\gamma}
\right\rfloor.
\end{equation*}
Consequently, $\bar n$ is the largest $n$ satisfying \eqref{Rbarn}.
Instead of the second condition in \eqref{jaannosehto} we put
\begin{equation}\label{Ncondition}
\log(2|N|) \ge \frac{K^2-n^2}{2}\log|d|,
\end{equation}
which is  equivalent to
\begin{equation*}\label{nylaraja}
n\le n_2:=\sqrt{\frac{T}{\gamma^2-1}}.
\end{equation*}
The assumption \eqref{lahella} is equivalent to
\begin{equation}\label{dsNds2}
|d|^s-1\le |N|\tau\le |d|^s+1.
\end{equation}
Suppose $|N|\ge 2/\tau$, then \eqref{dsNds2} and \eqref{Ncondition} imply
\begin{align*}\label{}
s&\ge \frac{\log(|N|\tau-1)}{\log|d|} \ge 
\frac{\log(2|N|)}{\log|d|}+\frac{\log(\tau/4)}{\log|d|}\\
&\ge\frac{K^2-n^2}{2}-\frac{K+n}{2}+\frac{K+n}{2}+\frac{\log(\tau/4)}{\log|d|}
\ge \frac{K^2-n^2}{2}-\frac{K+n}{2},
\end{align*}
if $n\ge \frac{2}{1+\gamma}\frac{\log(4/\tau)}{\log|d|}$.
So, the assumption \eqref{sKKnn} in Lemma \ref{denD}, see also \eqref{DnK}, is satisfied and thus 
\begin{equation}\label{ISONalaraja1}
D_{n,K}\in \mathbb{Z},\quad \text{if}\quad |N|\ge \frac{2}{\tau},\ 
n\ge \frac{2}{1+\gamma}\frac{\log(4/\tau)}{\log|d|}.
\end{equation}
Now we set
\begin{equation}\label{barngamman2}
\bar n+2\le n_2.
\end{equation}
Hence $(d^s,N)\in \mathcal{E}$, if the assumption \eqref{barngamman2} holds.
Finally, note that the condition \eqref{barngamman2} is equivalent to \eqref{gammaN}
and the conditions \eqref{Ncondition} and \eqref{ISONalaraja1} imply a lower bound
\begin{equation}\label{ISONalaraja2}
\log(2|N|)\ge \log(2N_2):=
\max\{\frac{2(\gamma-1)}{(\gamma+
1)\log|d|}\left(\log\frac{4}{\tau}\right)^2\}.
\end{equation}
\end{proof}

In Lemma \ref{EFFECT} the condition \eqref{gammaN} gives an upper bound, say $Y(|N|)$, for 
a feasible $\gamma$ to be used in \eqref{IRRATEXP1b}.
We will consider what happens when $|N|\to\infty$. 
Then we may use any value from the interval 
$1<\gamma< 1+\sqrt{3}=Y(\infty)$ and as a limit we get an optimal value  
\begin{equation*}
\frac{a(\gamma)}{b(\gamma)}\rightarrow 1+ \frac{1}{3+2\sqrt{3}}
\end{equation*}
for the main term.
Hence  
\begin{equation*}
2+ \frac{1}{3+2\sqrt{3}}
\end{equation*}
is an asymptotic irrationality exponent of $\tau$.

Let $q=1/d$, where $|d| \in \Z_{\ge 2}$, and $t=\frac{u}{v}\in\Q$, gcd$(u,v)=1$, $v\in \Z_+$, $0<|t|<1$. 
By Lemmas \ref{LBLEMMA} and \ref{EFFECT} we have a lower bound 
\begin{equation*}\label{}
\left| \tau-\frac{d^s}{N} \right| >\frac{1}{e^{c_3(\gamma)}}
 \frac{1}{(2|N|)^{1+a(\gamma)/b(\gamma)+c_2(\gamma)/\sqrt{b(\gamma)\log(2|N|)}}}
\end{equation*}
valid for any $\gamma$ satisfying \eqref{gammaN}.
Therefore 
\begin{equation*}\label{}
\left|E_q (t)- \frac{d^s}{N}\right|
 \ge C_2 (2|N|)^{-(2+\frac{1}{3+2\sqrt 3}+\varepsilon_2)}
\end{equation*}
holds for all $s\in \Z_+, N\in\Z$, $|N| \ge N_2$ with some constant $C_2=C_2(d,v)$ and
error term $\varepsilon_2=\varepsilon_2(d,v,N)\in\R_+$ satisfying 
$\varepsilon_2(d,v,N)\to 0$, when $|N|$ increases.
By using the values 
$\gamma=1+\sqrt{3}$, $\tau=1/5$ and $|d|=2$
in \eqref{ISONalaraja2} we get a numerical value
$\log(2N_2):=12.0177...$, thus we fix $N_2=82836$.\qed







\begin{thebibliography}{99}
\bibitem{andrews} G. E. Andrews, "The Theory of Partitions", Addison-Wesley, 1976.
\bibitem{bundshuh} P. Bundschuh, Ein Satz \" uber ganze Funktionen und Irrationalit\" atsaussagen, (German) Invent. Math. 9 1969/1970. 175-184.
\bibitem{lucmarsta} F. Luca, D. Marques, P. Stanica, On the spacings between C-nomial coeffients, Journal of Number Theory 130 (2010), 82-100.
\bibitem{matala} T. Matala-aho, Remarks on the aritmetic properties of certain hypergeometric series of Gauss and Heine, Acta Univ. Oulu. Ser. A Sci. Rerum Natur. 219 (1991), 1-112.
\bibitem{matvaa} T. Matala-aho, K. V\" a\" an\" anen, On Diophantine approximations of mock theta functions of third order, Ramanujan J. 4 (2000), 13-28.

\end{thebibliography}
\end{document}